\documentclass[12pt,leqno]{amsart}
\usepackage{verbatim,amssymb,amsfonts,latexsym,amsmath,amsthm,tikz}

\newtheorem{theorem}{Theorem}[section]
\newtheorem*{maintheorem}{Main Theorem}
\newtheorem*{maintheorem2}{Main Theorem (version two)}
\newtheorem*{thm}{Theorem~\ref{thm:external}}
\newtheorem*{thm2}{Theorem~\ref{thm:ch}}
\newtheorem*{thm3}{Theorem~\ref{thm:MA}}
\newtheorem*{thm4}{Theorem~\ref{thm:BS}}
\newtheorem{lemma}[theorem]{Lemma}
\newtheorem{corollary}[theorem]{Corollary}
\newtheorem{proposition}[theorem]{Proposition}
\newtheorem{question}[theorem]{Question}
\theoremstyle{definition}
\newtheorem{example}[theorem]{Example}

\theoremstyle{remark}

\newtheorem*{sub-claim}{sub-claim}

\newcommand{\Q}{\mathbb{Q}}

\newcommand{\N}{\mathbb{N}}
\newcommand{\Z}{\mathbb{Z}}

\newcommand{\explicitSet}[1]{\left\lbrace #1 \right\rbrace}
\newcommand{\brackets}[1]{\left\langle #1 \right\rangle}
\newcommand{\set}[2]{\explicitSet{#1 \colon #2}}
\newcommand{\seq}[2]{\brackets{#1 \colon #2}}
\newcommand{\<}{\langle}
\renewcommand{\>}{\rangle}
\newcommand{\0}{\emptyset}
\renewcommand{\a}{\alpha}
\renewcommand{\b}{\beta}
\newcommand{\g}{\gamma}
\renewcommand{\k}{\kappa}
\newcommand{\s}{\sigma}
\newcommand{\e}{\varepsilon}
\newcommand{\dlt}{\delta}
\newcommand{\w}{\omega}
\newcommand{\sub}{\subseteq}
\newcommand{\rest}{\!\restriction\!}

\newcommand{\closure}[1]{\overline{#1}}

\newcommand{\U}{\mathcal{U}}
\newcommand{\V}{\mathcal{V}}
\newcommand{\plim}{p\mbox{-}\!\lim_{n \in \w}}

\newcommand{\splim}{\s(p)\mbox{-}\!\lim_{n \in \w}}

\newcommand{\p}{\mathbb{P}}

\newcommand{\pwmf}{\mathcal{P}(\w)/\mathrm{fin}}
\newcommand{\continuum}{\mathfrak{c}}
\newcommand{\pseudo}{\mathfrak{p}}

\newcommand{\ch}{\mathrm{CH}}
\newcommand{\zfc}{\mathrm{ZFC}}

\newcommand{\ocama}{\mathrm{OCA+MA}}

\newcommand{\mask}{\mathrm{MA}_\k(\s\text{-centered})}

\begin{document}

\title{Shift-preserving maps on $\w^*$}
\author{Will Brian}
\address {
William R. Brian\\
Department of Mathematics\\
Baylor University\\
One Bear Place \#97328\\
Waco, TX 76798-7328}
\email{wbrian.math@gmail.com}

\thanks{$*$ The author would like to thank Brian Raines and Alan Dow for helpful conversations, and Brian especially for patiently listening to several preliminary (and wrong) ideas for proving the main theorem of this paper.}

\begin{abstract}
The shift map $\s$ on $\w^*$ is the continuous self-map of $\w^*$ induced by the function $n \mapsto n+1$ on $\w$. Given a compact Hausdorff space $X$ and a continuous function $f: X \to X$, we say that $(X,f)$ is a quotient of $(\w^*,\s)$ whenever there is a continuous surjection $Q: \w^* \to X$ such that $Q \circ \s = \s \circ f$.

Our main theorem states that if the weight of $X$ is at most $\aleph_1$, then $(X,f)$ is a quotient of $(\w^*,\s)$ if and only if $f$ is weakly incompressible (which means that no nontrivial open $U \sub X$ has $f(\closure{U}) \sub U$). Under $\ch$, this gives a complete characterization of the quotients of $(\w^*,\s)$ and implies, for example, that $(\w^*,\s^{-1})$ is a quotient of $(\w^*,\s)$.

In the language of topological dynamics, our theorem states that a dynamical system of weight $\aleph_1$ is an abstract $\w$-limit set if and only if it is weakly incompressible.

We complement these results by proving $(1)$ our main theorem remains true when $\aleph_1$ is replaced by any $\k < \pseudo$, $(2)$ consistently, the theorem becomes false if we replace $\aleph_1$ by $\aleph_2$, and $(3)$ $\ocama$ implies that $(\w^*,\s^{-1})$ is not a quotient of $(\w^*,\s)$.

\end{abstract}

\maketitle

\section{Introduction}

In \cite{Par}, Parovi\v{c}enko proved that every compact Hausdorff space of weight $\aleph_1$ is a continuous image of $\w^* = \b\w - \w$. In this paper we prove an analogous result concerning the continuous maps on $\w^*$ that respect the shift map.

The \emph{shift map} $\s: \b\w \to \b\w$ sends an ultrafilter $p$ to the unique ultrafilter generated by $\set{A+1}{A \in p}$. Equivalently, $\s$ is the unique map on $\b\w$ that continuously extends the map $n \mapsto n+1$ on $\w$. The shift map restricts to an autohomeomorphism of $\w^*$.

If $X$ is a compact Hausdorff space and $f: X \to X$ is continuous, we say that $(X,f)$ is a \emph{quotient} of $(\w^*,\s)$ whenever there is a continuous surjection $Q: \w^* \to X$ such that $Q \circ \s = f \circ Q$. The main theorem of this paper characterizes the quotients of $(\w^*,\s)$ that have weight at most $\aleph_1$:

\begin{maintheorem}
Suppose $X$ is a compact Hausdorff space with weight at most $\aleph_1$, and $f: X \to X$ is continuous. Then $(X,f)$ is a quotient of $(\w^*,\s)$ if and only if $f$ is weakly incompressible.
\end{maintheorem}

Recall that $f: X \to X$ is \emph{weakly incompressible} if for any open $U \sub X$ with $\0 \neq U \neq X$, we have $f(\closure{U}) \not\sub U$. This theorem is the appropriate analogue of Parovi\v{c}enko's because $(\w^*,\s)$ is itself weakly incompressible, and this property is always preserved by taking quotients. In other words, our theorem isolates a property of the shift map that determines exactly when Parovi\v{c}enko's topological result extends to a result of dynamics.

\subsection*{Connection with topological dynamics}

A \emph{dynamical system} is a pair $(X,f)$, where $X$ is a compact Hausdorff space and $f: X \to X$ is continuous. For example, $(\w^*,\s)$ is a dynamical system, and our main theorem states that it is universal (in the ``mapping onto'' sense) for all weakly incompressible dynamical systems of weight $\leq \aleph_1$.

Given a point $x \in X$, the $\w$\emph{-limit set of} $x$ is the set of all limit points of the orbit of $x$:
$$\w_f(x) = \bigcap_{n \in \w}\closure{\set{f^m(x)}{m \geq n}}.$$
It is easy to see that $\w_f(x)$ is closed under $f$, so that $(\w_f(x),f)$ is itself a dynamical system.

Recall that two dynamical systems $(X,f)$ and $(Y,g)$ are \emph{isomorphic} (or, for some authors, \emph{conjugate}) if there is a homeomorphism $H: X \to Y$ with $H \circ f = g \circ H$. An \emph{abstract $\w$-limit set} is a dynamical system that is isomorphic to a dynamical system of the form $(\w_f(x),f)$.

For example, $(\w^*,\s)$ is an abstract $\w$-limit set because $\w^* = \w_\s(n)$ for any $n \in \w$ in the larger dynamical system $(\b\w,\s)$. Notice that $\w^*$ is not an $\w$-limit set ``internally''; that is, $\w^* \neq \w_\s(p)$ for any $p \in \w^*$ (indeed, $\w^*$ is not even separable). In order to realize $(\w^*,\s)$ as an $\w$-limit set, it is necessary to extend it to a larger dynamical system.

In the next section, we will prove the following characterization of abstract $\w$-limit sets:

\begin{thm}
$(X,f)$ is an abstract $\w$-limit set if and only if it is a quotient of $(\w^*,\s)$.
\end{thm}

In other words, $(\w^*,\s)$ is universal among all abstract $\w$-limit sets. Theorem~\ref{thm:external} is one of the motivations for our study of the quotients of $(\w^*,\s)$: it is the study of the internal structure of $\w$-limit sets.

%Interestingly, Stefan Geschke proves in \cite{Ges} that no metrizable dynamical system is universal for the metrizable abstract $\w$-limit sets. 

Theorem~\ref{thm:external} allows us to rephrase our main theorem as follows:
\begin{maintheorem2}
Suppose $(X,f)$ is a dynamical system and the weight of $X$ is at most $\aleph_1$. $(X,f)$ is an abstract $\w$-limit set if and only if $f$ is weakly incompressible.
\end{maintheorem2}
This way of stating the main theorem reveals it as an extension of the following well-known result of Bowen and Sharkovsky:

\begin{thm4}
A metrizable dynamical system is an abstract $\w$-limit set if and only if it is weakly incompressible.
\end{thm4}

Sharkovsky proves the forward direction in \cite{Srk} and Bowen proves the converse in \cite{Bwn}. We will give a slightly different proof below, because we will require a mild strengthening of this theorem (Corollary~\ref{cor:BS}) to prove our main result. See \cite{BGOR} or \cite{M&R}, and the references therein, for further research on the connection between weak incompressibility and $\w$-limit sets.

\subsection*{Outline of the proof}

Of the various proofs of Parovi\v{c}enko's theorem, ours is closest in spirit to that of B{\l}aszczyk and Szyma\'nski in \cite{B&S}. Their proof begins by writing a given compact Hausdorff space $X$ as a length-$\omega_1$ inverse limit of compact metrizable spaces: $X = \varprojlim \seq{X_\a}{\a < \w_1}$. They then construct a coherent transfinite sequence of continuous surjections $Q_\a: \w^* \to X_\a$, and define $Q: \w^* \to X$ to be the inverse limit of this sequence. The $Q_\a$ are constructed recursively, using a variant of the following lifting lemma at successor stages:

\begin{lemma}\label{lem:easylift}
Let $Y$ and $Z$ be compact metrizable spaces, and let $Q_Z: \omega^* \to Z$ and $\pi: Y \to Z$ be continuous surjections. Then there is a continuous surjection $Q_Y: \w^* \to Y$ such that $Q_Z = \pi \circ Q_Y$.
\end{lemma}

In our situation, the first part of B{\l}aszczyk and Szyma\'nski's proof goes through: we prove in Corollary~\ref{cor:inverselimit} below that given a dynamical system $(X,f)$ of weight $\aleph_1$, one may always write $(X,f)$ as a length-$\w_1$ inverse limit of metrizable dynamical systems. However, we run into trouble with the analogue of Lemma~\ref{lem:easylift}: the analogous lemma for dynamical systems is false (see Example~\ref{ex:liftingproblems}).

To get around this problem, we modify B{\l}aszczyk and Szyma\'nski's approach by using sharper tools. Rather than beginning with $(X,f)$ and writing it as a topological inverse limit, we begin with a particular embedding of $X$ in $[0,1]^{\w_1}$ and use a much stronger form of inverse limit: a continuous chain of elementary submodels of a sufficiently large fragment of the set-theoretic universe. Each model in our chain naturally gives rise to a metrizable ``reflection'' of $(X,f)$, and the continuity requirement organizes these reflections into an inverse limit system with limit $(X,f)$. Elementarity gives this system strong structural properties, and ultimately is the key that unlocks a workable analogue of Lemma~\ref{lem:easylift}.

Our use of elementarity is inspired by the work of Dow and Hart in \cite{D&H}, where they prove that every continuum of weight $\aleph_1$ is a continuous image of $\mathbb{H}^*$, the Stone-\v{C}ech remainder of $\mathbb{H} = [0,\infty)$. They give three proofs of this fact, each of which relies on elementarity in some essential way. The proof of our main theorem is most similar to their third proof, found in Section 3 of \cite{D&H}.

In Section~\ref{sec:MA}, we will show that both Parovi\v{c}enko's theorem about continuous images of $\w^*$ and the Dow-Hart theorem about continuous images of $\mathbb H^*$ can be derived as relatively straightforward corollaries of our main theorem. In light of this, it is unsurprising that our proof uses some of the same ideas found in \cite{B&S} and \cite{D&H}.

\subsection*{Extensions and limitations}

Under the Continuum Hypothesis, our result gives a complete characterization of the quotients of $(\w^*,\s)$:

\begin{thm2}
Assuming $\ch$, the following are equivalent:
\begin{enumerate}
\item $(X,f)$ is a quotient of $(\w^*,\s)$.
\item $X$ has weight at most $\continuum$ and $f$ is weakly incompressible.
\item $X$ is a continuous image of $\w^*$ and $f$ is weakly incompressible.
\end{enumerate}
\end{thm2}

Every quotient of $(\w^*,\s)$ is weakly incompressible, so $(3)$ gives the most liberal possible characterization of quotients of $(\w^*,\s)$: they are simply the weakly incompressible dynamical systems for which the topology is not an obstruction.

In Section~\ref{sec:MA}, we show that the nontrivial conclusions of Theorem~\ref{thm:ch} are independent of $\zfc$. Specifically, we show that $(2)$ does not imply $(1)$ or $(3)$ in the Cohen model, and that $(3)$ does not imply $(1)$ under $\ocama$. In fact, we will show under $\ocama$ that $(\w^*,\s^{-1})$ is not a quotient of $(\w^*,\s)$, even though $\s^{-1}$ is weakly incompressible.

We also show in Section~\ref{sec:MA} that if $\k < \pseudo$ then our main theorem holds with $\k$ in the place of $\aleph_1$:
\begin{thm3}
If the weight of $X$ is less than $\pseudo$, then $(X,f)$ is a quotient of $(\w^*,\s)$ if and only if $f$ is weakly incompressible.
\end{thm3}
In the same way that our main theorem is the dynamical analogue of Parovi\v{c}enko's theorem, this result is the dynamical analogue of the following result of van Douwen and Przymusi\'nski from \cite{vDP}: \emph{If $X$ is a compact Hausdorff space with weight less than $\pseudo$, then $X$ is a continuous image of $\w^*$}.

%\subsection*{Organization of the paper}

%In Section~\ref{sec:background} we will review some facts about $\b\w$, and flesh out the connections between our main result and the theory of abstract $\w$-limit sets. In Section~\ref{sec:lemmas} we prove several technical lemmas leading up to the proof of the main theorem in Section~\ref{sec:main}. We have done our best to present the heart of the proof in Section~\ref{sec:main}, and to remove any distractions to Sections \ref{sec:background} and \ref{sec:lemmas}. Section~\ref{sec:MA} contains some negative results and the extension of our theorem to cardinals $\k < \pseudo$.

\section{First steps}\label{sec:background}

\subsection*{Extending maps from $\w$ to $\b\w$}

If $X$ is a compact Hausdorff space and $f: \w \to X$ is any function, then there is a unique continuous function $\b f: \b\w \to X$ that extends $f$, the \emph{Stone extension} of $f$. For a sequence $\seq{x_n}{n \in \N}$ of points in $X$ and $p \in \b\w$, we will usually write $\plim x_n$ for the image of $p$ under the Stone extension of the function $n \mapsto x_n$. We will need the following facts about Stone extensions (proofs can be found in chapter 3 of \cite{H&S}):

\begin{lemma}\label{lem:plimits}
Let $X$ be a compact Hausdorff space and $\seq{x_n}{n < \w}$ a sequence of points in $X$.
\begin{enumerate}
\item $\plim x_n = y$ if and only if for every open $U \ni y$ we have $\set{n}{x_n \in U} \in p$
\item $p \mapsto \plim x_n$ is a continuous function $\b\w \to X$.
\item If $f: X \to X$ is continuous and $p \in \b\w$, then
$$\textstyle f \!\left( \plim x_n \right) = \plim f(x_n).$$
\item For each $p \in \b\w$, $\splim x_n = \plim x_{n+1}$.
\end{enumerate}
\end{lemma}

\subsection*{Extending maps from $\w^*$ to $\b\w$}

The following folklore result is a fairly straightforward consequence of the Tietze Extension Theorem, or an alternative proof can be found in \cite{Eng}, Theorem 3.5.13.

\begin{lemma}\label{lem:metricextensions}
Suppose $X$ is a compact Hausdorff space and $f: \w^* \to X$ is continuous. There is a compact Hausdorff space $Y \supseteq X$, such that $f$ can be extended to a continuous function $F: \b\w \to Y$. Furthermore, we may assume that $F \rest \w$ is injective, and that $F(\w)$ is an open, relatively discrete subset of $Y$ with $F(\w) \cap X = \0$.
\end{lemma}
%\begin{proof}
%Because $X$ is a compact Hausdorff space, we may assume that $X$ is a closed subset of $[0,1]^\k$ for some $\k$. The Tietze Extension Theorem guarantees that $f$ can be extended to a map $F_0: \b\w \to [0,1]^\k$, which proves the first assertion. To get the second, replace $[0,1]^\k$ with $[0,1] \times [0,1]^\k$, identify $X$ with $\{0\} \times X$, and modify $F_0$ as follows:
%$$
%F(p) = \begin{cases} 
%(\frac{1}{2^p},F_0(p)) & \textrm{ if $p \in \w$} \\
%(0,F_0(p)) & \textrm{ if $p \in \w^*$}.
%\end{cases}
%$$
%This function has the desired properties.
%\end{proof}

\begin{lemma}\label{lem:continuity}
Let $(X,f)$ be a dynamical system, and $Q: X \to Y$ a continuous surjection such that, for all $x_1,x_2 \in X$, if $Q(x_1) = Q(x_2)$ then $Q(f(x_1)) = Q(f(x_2))$. Then there is a unique continuous $g: Y \to Y$ such that $g \circ Q = Q \circ f$.
\end{lemma}
\begin{proof}
The assumptions about $Q$ immediately imply that there is a unique function $g: Y \to Y$ such that $g \circ Q = Q \circ f$, namely $g(y) = Q(f(Q^{-1}(y)))$. We need to check that this function is continuous.

If $K$ is a closed subset of $Y$, then $f^{-1}(Q^{-1}(K))$ is closed in $X$. Because $X$ is compact, $f^{-1}(Q^{-1}(K))$ is compact, which implies $g^{-1}(K) = Q(f^{-1}(Q^{-1}(K)))$ is closed. Since $K$ was arbitrary, $g$ is continuous.
\end{proof}

In the same way that our main theorem can be seen as a dynamical version of Parovi\v{c}enko's theorem, the following result can be seen as a dynamical version of Lemma~\ref{lem:metricextensions}:

\begin{theorem}\label{thm:external}
$(X,f)$ is an abstract $\w$-limit set if and only if it is a quotient of $(\w^*,\s)$.
\end{theorem}
\begin{proof}
It is well known that if $(X,f)$ is an $\w$-limit set then it is a quotient of $(\w^*,\s)$. Indeed, the map $p \mapsto \plim f^n(x)$ gives a quotient mapping from $(\w^*,\s)$ to $(\w_f(x),f)$. For details and some discussion, see Section 2 of \cite{Bls}. Here we need to prove the converse.

Suppose $q: \w^* \to X$ is a quotient mapping from $(\w^*,\s)$ to $(X,f)$. Using Lemma~\ref{lem:metricextensions}, there is a compact Hausdorff space $Y$ containing $X$ such that $q$ extends to a continuous function $Q: \b\w \to Y$, where $Q \rest \w$ is injective, $Q(\w)$ is an open, relatively discrete subset of $Y$, and $Q(\w) \cap X = \0$. Replacing $Y$ with $Q(\b\w)$ if necessary, we may also assume that $Q$ is surjective.

Define $g: Y \to Y$ by
$$
g(y) = \begin{cases} 
f(y) & \textrm{ if $y \in X$} \\
Q(n+1) & \textrm{ if $y = Q(n)$, $n \in \w$}
\end{cases}
$$
This function is well-defined because $Q \rest \w$ is injective and $Q(\w) \cap X = \0$. By design, $Q \circ \s = g \circ Q$. By Lemma~\ref{lem:continuity}, $g$ is continuous.

To finish the proof, we will show that, in $(Y,g)$, $X$ is an $\w$-limit set. Letting $p = Q(0)$, we claim $X = \w_g(p)$. Notice that
$$\set{g^m(Q(0))}{m \geq n} = \set{Q(m)}{m \geq n}$$
for all $n$. Using the continuity of $Q$, we have
$$\closure{\set{Q(m)}{m \geq n}} \supseteq Q(\closure{\set{m}{m \geq n}}) \supseteq Q(\w^*) = X.$$
Thus $\w(Q(0)) \supseteq X$. The reverse inclusion follows from the fact that $Q(\w)$ is open and relatively discrete.
\end{proof}

\subsection*{Chain transitivity}

Suppose $(X,f)$ is a dynamical system and $d$ is a metric for $X$. An $\e$\emph{-chain} in $(X,f)$ is a sequence $\seq{x_i}{i \leq n}$ such that $d(f(x_i),x_{i+1}) < \e$ for all $i < n$. Roughly, an $\e$-chain is a piece of an orbit, but computed with a small error at each step. $(X,f)$ is called \emph{chain transitive} if for any $a,b \in X$ and any $\e > 0$, there is an $\e$-chain beginning at $a$ and ending at $b$.

Using open covers in the place of $\e$-balls, we can reformulate the definition of chain transitivity so that it applies to non-metrizable dynamical systems. Given $(X,f)$ and an open cover $\U$ of $X$, we say that $\seq{x_i}{i \leq n}$ is a $\U$\emph{-chain} if, for every $i < n$, there is some $U \in \U$ such that $f(x_i) \in U$ and $x_{i+1} \in U$. A dynamical system $(X,f)$ is \emph{chain transitive} if for any $a,b \in X$ and any open cover $\U$ of $X$, there is a $\U$-chain beginning at $a$ and ending at $b$.

\begin{lemma}\label{lem:ct}$\ $
\begin{enumerate}
\item A dynamical system is chain transitive if and only if it is weakly incompressible.
\item Every quotient of $(\w^*,\s)$ is weakly incompressible.
\end{enumerate}
\end{lemma}

The proof of $(1)$ is essentially the same as the proof for metrizable dynamical systems (see, e.g., Theorem 4.12 in \cite{Akn}). Both $(1)$ and $(2)$ can be found (with proofs) in Section 5 of \cite{WRB}.

\subsection*{The Bowen-Sharkovsky theorem}

We now give a proof of the theorem of Bowen and Sharkovsky mentioned in the introduction.

\begin{theorem}[Bowen-Sharkovsky]\label{thm:BS}
A metrizable dynamical system is an abstract $\w$-limit set if and only if it is weakly incompressible.
\end{theorem}
\begin{proof}
The forward direction is a consequence of Theorem~\ref{thm:external} and Lemma~\ref{lem:ct}. To prove the reverse direction, we will use chain transitivity instead of weak incompressibility.

Let $(X,f)$ be a chain transitive dynamical system, and let $d$ be a metric for $X$. Pick $x_0 \in X$ arbitrarily. Using chain transitivity and the compactness of $X$, define $x_1, x_2, \dots, x_{n_1}$ so that 
\begin{enumerate}
\item $\seq{x_i}{0 \leq i \leq n_1}$ is a $1$-chain
\item $\bigcup_{0 \leq i \leq n_1} B_1(x_i) = X$, and 
\item $x_{n_1} = x_0$.
\end{enumerate}
Now assuming that $\seq{x_i}{i \leq n_m}$ have already been defined, define $x_{n_m+1}, x_{n_m+2}, \dots, x_{n_{m+1}}$ so that
\begin{enumerate}
\item $\seq{x_i}{n_m \leq i \leq n_{m+1}}$ is a $\frac{1}{m}$-chain,
\item $\bigcup_{n_m \leq i \leq n_{m+1}} B_{\frac{1}{m}}(x_i) = X$, and 
\item $x_{n_{m+1}} = x_0$.
\end{enumerate}
It is not difficult to see that chain transitivity and compactness together allow us to build such a sequence of points.

Define $Q: \w^* \to X$ to be the Stone extension of the map $n \mapsto x_n$. This function is automatically continuous. It follows from $(2)$ above that $\set{x_m}{m \geq n}$ is dense in $X$ for every $n$, which implies that $Q$ is surjective. It remains to show that $Q \circ \s = f \circ Q$.

Fix $p \in \w^*$ and $\e > 0$. Let $m$ be sufficiently large (precisely, we will need $m > n_k$ where $\frac{1}{k} < \e$). Notice that
$$\textstyle Q(\s(p)) = \splim x_n = \plim x_{n+1}, \ \text{ and}$$
$$f(Q(p)) = f(\plim x_n) = \plim f(x_n).$$
Using the fact that $p$ is non-principal and that $d(f(x_n),x_{n+1}) < \e$ for every $n \geq m$, we have
$$\textstyle d(f(Q(p)),Q(\s(p))) = d(\plim f(x_n),\plim x_{n+1}) \leq \e.$$
Since $\e$ was arbitrary, $f(Q(p)) = Q(\s(p))$. Since $p$ was also arbitrary, $Q \circ \s = f \circ Q$ as desired.
\end{proof}

After developing a few more definitions in the next section, we will state a slightly stronger version of this result (which already follows from the given proof). This stronger version will be the base step in our recursive proof of the main theorem.

\section{A few lemmas}\label{sec:lemmas}

In this section we begin the proof of our main theorem in the form of several lemmas. The heart of the proof -- a transfinite recursion driven by a chain of elementary submodels -- will be in the next section.

Given an ordinal $\dlt$, the \emph{standard basis} for $[0,1]^\dlt$ is the basis generated by sets of the form $\pi_\a^{-1}(p,q)$, where $p,q \in [0,1] \cap \Q$ and $\pi_\a$ is the projection mapping a point of $[0,1]^\dlt$ to its $\a^{\mathrm{th}}$ coordinate. Whenever we mention basic open subsets of $[0,1]^\dlt$, this is the basis we mean. Notice that every basic open subset of $[0,1]^\dlt$ can be defined using finitely many ordinals less than $\dlt$ and finitely many rational numbers.

Suppose $X$ is a closed subset of $[0,1]^\dlt$. By an \emph{open cover} of $X$, we will mean a set $\mathcal U$ of open subsets of $[0,1]^\dlt$ with $X \sub \bigcup \mathcal U$. A \emph{nice open cover} of $X$ is a finite open cover $\U$ of $X$ consisting of basic open subsets of $[0,1]^\dlt$, such that $U \cap X \neq \0$ for all $U \in \U$.

If $\mathcal U$ is a collection of subsets of $[0,1]^\dlt$ and $A \sub [0,1]^\dlt$,
$$\mathcal U_\star(A) = \bigcup \set{U \in \mathcal U}{U \cap A \neq \0}.$$
For convenience, if $A = \{a\}$ we write $\mathcal U_\star(a)$ instead of $\mathcal U_\star(\{a\})$.

If $\mathcal U$ and $\mathcal V$ are collections of open sets, recall that $\mathcal U$ \emph{refines} $\mathcal V$ if for every $U \in \mathcal U$ there is some $V \in \mathcal V$ with $U \sub V$. $\mathcal U$ is a \emph{star refinement} of $\mathcal V$ if for every $U \in \mathcal U$ there is some $V \in \mathcal V$ such that $\mathcal U_\star(U) \sub V$. It is known (see, e.g., Theorem 5.1.12 in \cite{Eng}) that every open cover of a compact Hausdorff space has a star refinement.

\begin{lemma}\label{lem:stars}
Let $X$ be a closed subset of $[0,1]^\dlt$. A function $f: X \to X$ is continuous if and only if for every open cover $\mathcal U$ of $X$ there is a nice open cover $\mathcal V$ of $X$ such that
$$\set{\mathcal V_\star(f(\mathcal V_\star(x) \cap X))}{x \in X}$$
is an open cover of $X$ that refines $\mathcal U$.
\end{lemma}
\begin{proof}
Suppose that $f$ is continuous and let $\mathcal U$ be an open cover of $X$. Let $\mathcal W$ be a star refinement of $\U$. By continuity, $f^{-1}(W \cap X)$ is a relatively open subset of $X$ for every $W \in \mathcal W$. For each $W \in \mathcal W$ pick some open subset $W^\leftarrow$ of $[0,1]^\dlt$ such that $W^\leftarrow \cap X = f^{-1}(W \cap X)$. Let $\mathcal W^\leftarrow = \set{W^\leftarrow}{W \in \mathcal W}$, and observe that $\mathcal W^\leftarrow$ is an open cover of $X$. Let $\mathcal Y$ be a star refinement of $\mathcal W^\leftarrow$, and let $\mathcal V$ be a common refinement of $\mathcal Y$ and $\mathcal W$, for example $\set{Y \cap W}{Y \in \mathcal Y \text{ and } W \in \mathcal W}$. By refining $\V$ further we may assume it consists of basic open sets; by throwing some sets away we may assume $\V$ is finite and every element of $\V$ meets $X$. In other words, we may take $\V$ to be a nice open cover.

If $x \in X$, then $\V_\star(x) \sub \mathcal Y_\star(x) \sub W^\leftarrow$ for some $W^\leftarrow \in \mathcal W^\leftarrow$. By the definition of $\mathcal W^\leftarrow$, $f(\V_\star(x) \cap X) \sub W$ for some $W \in \mathcal W$. Because $\V$ refines $\mathcal W$ and $\mathcal W$ star refines $\mathcal U$, $\V_\star(f(\V_\star(x) \cap X)) \sub \mathcal W_\star(W) \sub U$ for some $U \in \U$.

For the other direction, suppose that $f$ satisfies the conclusion of the lemma. Fix $x \in X$ and let $U,W$ be open sets containing $f(x)$ such that $f(x) \in W \sub \closure{W} \sub U$. Let $\mathcal U = \{U,[0,1]^\k - W\}$, and let $\V$ be a nice open cover of $X$ satisfying the conclusion of the lemma. Setting $V = \V_\star(f(\V_\star(x) \cap X))$, we must have either $V \sub U$ or $V \cap U = \0$. The latter is impossible because $f(x) \in V$, so $V \sub U$. Thus we have found a neighborhood of $x$ in $X$, namely $\V_\star(x) \cap X$, whose image under $f$ is contained in $U \cap X$. Since $U$ and $x$ were arbitrary, $f$ is continuous.
\end{proof}

Given a countable ordinal $\dlt$, define $\Pi_\dlt: [0,1]^{\w_1} \to [0,1]^\dlt$ to be the natural projection onto the first $\dlt$ coordinates, namely $\Pi_\dlt = \Delta_{\a < \dlt}\pi_\a$.

\begin{lemma}\label{lem:projections}
Let $X$ be a closed subset of $[0,1]^{\w_1}$ and let $f: X \to X$ be continuous. There is a closed unbounded $C \sub \w_1$ such that for every $\dlt \in C$ and $x,y \in X$, if $\Pi_\dlt(x) = \Pi_\dlt(y)$ then $\Pi_\dlt(f(x)) = \Pi_\dlt(f(y))$.
\end{lemma}
\begin{proof}
For each $\a < \w_1$, let $\mathcal N_\a$ denote the set of all nice open covers of $X$ that are defined using only ordinals less than $\a$.

For each open cover $\U$ of $X$ there is a nice open cover $\V$ satisfying the conclusion of Lemma~\ref{lem:stars}, and $\V \in \mathcal N_\a$ for some $\a < \w_1$. For each $\a < \w_1$ define $\phi(\a)$ to be the least ordinal with the property that if $\U \in \mathcal N_\a$, then some $\V \in \mathcal N_{\phi(\a)}$ satisfies the conclusion of Lemma~\ref{lem:stars}. Because each $\mathcal N_\a$ is countable, $\phi$ maps countable ordinals to countable ordinals.

Let $C$ be the set of closure points of $\phi$:
$$C = \set{\dlt < \w_1}{\text{if } \a < \dlt \text{ then } \phi(\a) < \dlt}.$$
We claim that $C$ satisfies the conclusions of the lemma.

Suppose $\dlt \in C$ and $\Pi_\dlt(f(y)) \neq \Pi_\dlt(f(z))$. We may find some $\U \in \mathcal N_\dlt$ such that $\U$ separates $f(y)$ from $f(z)$, in the sense that there is no $U \in \U$ containing both $f(y)$ and $f(z)$. Because $\dlt$ is a closure point of $\phi$, there is some $\V \in \mathcal N_\dlt$ satisfying the conclusion of Lemma~\ref{lem:stars}.

For every $x \in X$, there is some $U \in \U$ such that$\V_\star(f(\V_\star(x) \cap X)) \sub U$. Because $f(y) \in f(\V_\star(y) \cap X)$ and $f(z) \in (\V_\star(z) \cap X)$, our choice of $\U$ guarantees $f(\V_\star(y) \cap X) \cap f(\V_\star(y) \cap X) = \0$, which implies $\V_\star(y) \cap \V_\star(z) \cap X = \0$. Since $\V \in \mathcal N_\dlt$, this implies $\Pi_\dlt(y) \neq \Pi_\dlt(z)$.
\end{proof}

\begin{corollary}\label{cor:inverselimit}
Every dynamical system of weight $\aleph_1$ can be written as an inverse limit of metrizable dynamical systems.
\end{corollary}
\begin{proof}
Let $(X,f)$ be a dynamical system of weight $\aleph_1$. Embed $X$ in $[0,1]^{\w_1}$, and let $C$ be the closed unbounded set of ordinals described in the previous lemma. For each $\dlt \in C$, let $X_\dlt = \Pi_\dlt(X)$ and define $f_\dlt: X_\dlt \to X_\dlt$ by $f_\dlt(\Pi_\dlt(x)) = \Pi_\dlt(f(x))$, which is continuous by Lemma~\ref{lem:continuity}. Then $\seq{(\Pi_\dlt(X),f_\dlt)}{\dlt \in C}$ is an inverse limit system, having the natural projections as bonding maps, and the limit of this system is $(X,f)$.
\end{proof}

Before moving on to our next lemma, we take a moment to justify the use of elementary submodels in the next section. Na\"{i}vely, one might wonder why we cannot simply prove our main theorem in the style of B{\l}aszczyk and Szyma\'nski, using Corollary~\ref{cor:inverselimit} and the appropriate analogue of Lemma~\ref{lem:easylift}:
\begin{itemize}
\item[($*$)] \emph{Let $(Y,g)$ and $(Z,h)$ be metrizable dynamical systems, and let $Q_Z: \omega^* \to Z$ and $\pi: Y \to Z$ be quotient mappings. Then there is a quotient mapping $Q_Y: \w^* \to Y$ such that $Q_Z = \pi \circ Q_Y$.}
\end{itemize}
The following example shows that $(*)$ is not true, so that we need more than a simple topological inverse limit structure in order to make B{\l}aszczyk and Szyma\'nski's proof go through. We will simply sketch the example and leave detailed proofs to the reader.

\begin{example}\label{ex:liftingproblems}
$([0,1],\mathrm{id})$ is a weakly incompressible dynamical system, and for our example it will play the role of $(Y,g)$ and $(Z,h)$ in $(*)$. Define $\pi: [0,1] \to [0,1]$ by setting $\pi(0) = 0$, $\pi(\frac{2}{3}) = 1$, and $\pi(1) = \frac{1}{2}$, and then extending $\pi$ linearly on the rest of $[0,1]$. We will define a quotient mapping $\pi_Z$ from $(\w^*,\s)$ to $([0,1],\mathrm{id})$ that does not lift through $\pi$.

Define $p_Z: \w \to [0,1]$ so that $p_Z(n)$ is the distance from $s(n) = \sum_{m \leq n}\frac{1}{m}$ to the nearest even integer ($s$ could be replaced with any increasing unbounded sequence of reals where the distance between successive terms goes to $0$). Letting $\pi_Z: \w^* \to [0,1]$ be the map induced by $p_Z$, it is easy to check (either directly, or using Lemma~\ref{lem:mainlemma} below) that $\pi_Z$ is a quotient mapping from $(\w^*,\s)$ to $([0,1],\mathrm{id})$.

Suppose $\pi_Y: \w^* \to [0,1]$ is another quotient mapping from $(\w^*,\s)$ to $([0,1],\mathrm{id})$. By the Tietze Extension Theorem, $\pi_Y$ is induced by a map $p_Y: \w \to [0,1]$. Suppose $\pi_Z = \pi \circ \pi_Y$. Then we must have $\lim_{n \to \infty} |p_Z(n) - \pi(p_Y(n))| = 0$. Since $\pi_Y \circ \s = \pi_Y$, we also must have $\lim_{n \to \infty}|p_Y(n) - p_Y(n+1)| = 0$. Putting these facts together, one may show that, for large enough $n$, $p_Y(n) \in [0,\frac{2}{3}+\varepsilon)$ for any prescribed $\e > 0$. This contradicts the surjectivity of $\pi_Y$. $\hfill \square$
\end{example}

Suppose $X \sub [0,1]^\dlt$ and $f: X \to X$ is continuous. If $\mathcal U$ is a nice open cover of $X$, we say that a sequence $\seq{x_n}{n < \w}$ is \emph{eventually compliant with} $\mathcal U$ if there exists some $m \in \w$ such that
\begin{enumerate}
\item $\set{x_n}{n \geq m} \sub \bigcup \mathcal U$,
\item $\set{x_n}{n \geq m} \cap U$ is infinite for all $U \in \U$, and
\item for all $n \geq m$, we have $x_{n+1} \in \mathcal U_\star(f(\mathcal U_\star(x_n) \cap X))$.
\end{enumerate}
Roughly, the idea behind this definition is that if our vision is blurred (with the amount of blurriness prescribed by $\mathcal U$), then $(1)$ it appears that every $x_n$ could be in $X$, $(2)$ it appears that $\set{x_n}{n \geq m}$ could be dense in $X$, and $(3)$ for each $n$, not only does it seem that $x_n$ could be in $X$, but also that $x_{n+1} = f(x_n)$.

\begin{lemma}\label{lem:mainlemma}
Let $X$ be a closed subset of $[0,1]^\dlt$ and let $f: X \to X$ be continuous. If $\seq{x_n}{n < \w}$ is a sequence of points in $[0,1]^\dlt$ that is eventually compliant with every nice open cover of $X$, then the map $p \mapsto \plim x_n$ is a quotient mapping from $(\w^*,\s)$ to $(X,f)$.

Conversely, if $Q$ is a quotient mapping from $(\w^*,\s)$ to $(X,f)$, then there is a sequence $\seq{x_n}{n < \w}$ in $[0,1]^\dlt$ such that $Q(p) = \plim x_n$ for all $p \in \w^*$, and this sequence is eventually compliant with every nice open cover of $X$.
\end{lemma}
\begin{proof}
Fix $X \sub [0,1]^\dlt$ and $f: X \to X$, and suppose $\seq{x_n}{n < \w}$ is a sequence of points in $[0,1]^\dlt$ that is eventually compliant with every nice open cover of $X$. Define $Q: \w^* \to [0,1]^\dlt$ by $Q(p) = \plim x_n$. From the definitions, we know that $Q$ is a continuous function with domain $\w^*$. We need to check that $Q(\w^*) = X$ and that $Q \circ \s = f \circ Q$.

First we show that $Q(\w^*) \sub X$. Let $U$ be any open subset of $[0,1]^\dlt$ containing $X$. There is some nice open cover $\mathcal U$ of $X$ such that $\bigcup \mathcal U \sub U$. By part $(1)$ of our definition of eventual compliance, $\plim x_n \in \closure{U}$ for every $p \in \w^*$. Since $U$ was arbitrary, $F(\w^*) \sub X$.

Next we show that $X \sub Q(\w^*)$. Let $U$ be any basic open subset of $[0,1]^\dlt$ with $U \cap X \neq \0$. We may find a nice open cover $\mathcal U$ of $X$ such that $U \in \mathcal U$. By part $(2)$ of the definition of eventual compliance, $Q(\w^*) \cap U \neq \0$. Because $Q(\w^*)$ is the continuous image of a compact space, and therefore closed, this shows $X \sub Q(\w^*)$. 

Lastly, we show that $Q \circ \s = f \circ Q$. Fix $p \in \w^*$, and let $U$ be an open neighborhood of $f(Q(p))$. We may find an open cover $\U$ of $X$ such that $U \in \U$ and $U$ is the only member of $\U$ containing $f(Q(p))$. Applying Lemma~\ref{lem:stars}, we obtain a nice open cover $\V$ of $X$ such that $\V_\star(f(\V_\star(Q(p)) \cap X)) \sub U$.

Let $m$ be large enough to witness the fact that $\seq{x_n}{n < \w}$ is eventually compliant with $\V$. Because $p$ is non-principal,
$$A = \set{n \geq m}{x_n \in \V_\star(Q(p))} \in p.$$
Using part $(3)$ of the definition of eventual compliance, $x_{n+1} \in U$ for every $n \in A$. Thus
$$\textstyle Q(\s(p)) = \splim x_n = \plim x_{n+1} \in \closure{U}.$$
Because $U$ was an arbitrary open neighborhood of $f(Q(p))$, this shows $Q(\s(p)) = f(Q(p))$. Since $p$ was arbitrary, $Q \circ \s = f \circ Q$ as desired. This finishes the proof of the first assertion of the lemma.

For the converse direction, suppose $Q$ is a quotient mapping from $(\w^*,\s)$ to $(X,f)$. By the Tietze Extension Theorem, $Q$ extends to a continuous function on $\b\w$. In other words, there is a sequence $\seq{x_n}{n < \w}$ of points in $[0,1]^\dlt$ such that $Q(p) = \plim x_n$ for every $p \in \w^*$. We want to show that this sequence is eventually compliant with every nice open cover of $X$. Using the fact that $Q(\w^*) = X$, it is easy to check parts $(1)$ and $(2)$ of the definition of eventual compliance.

To verify $(3)$, let $\U$ be a nice open cover of $X$ and suppose $\seq{x_n}{n < \w}$ is not eventually compliant with $\U$. Then there is an infinite $A \sub \w$ such that, for every $a \in A$, $x_{a+1} \notin \U_\star(f(\U_\star(x_a) \cap X))$. Let $p \in A^*$, let $x = Q(p)$, and fix $U \in \U$ with $x \in U$. By definition, $x = \plim x_n \in U$ implies that for some infinite $B \in p$, $\set{x_n}{n \in B} \sub U$. Replacing $B$ with $B \cap A$ if necessary, we may assume $B \sub A$. $B+1 \in \s(p)$, and for all $b \in B$ we have $x_{b+1} \notin \U_\star(f(\U_\star(x_b) \cap X)) \supseteq \U_\star(f(U \cap X))$. Thus
$$Q(\s(p)) = \splim x_n = \plim x_{n+1} \notin \U_\star(f(U \cap X)) \ni f(Q(p)).$$
Thus $Q \circ \s(p) \neq f \circ Q(p)$, contradicting the assumption that $Q$ is a quotient mapping.
\end{proof}

The next two definitions describe a particular kind of eventually compliant sequence, one that has been constructed in a certain way. These are the kinds of sequences that will be used in the next section.

As before, suppose $X$ is a closed subspace of $[0,1]^\dlt$ and that $f: X \to X$ is continuous. Given a nice open cover $\mathcal U$ of $X$ and a fixed point $x \in X$, we say that a finite sequence $\seq{x_i}{m < i \leq n}$ is a \emph{$\mathcal U$-compliant $x$-loop} provided
\begin{enumerate}
\item $x_n = x$,
\item $\set{x_i}{m < i \leq n} \sub \bigcup \mathcal U$,
\item $\set{x_i}{m < i \leq n} \cap U \neq \0$ for all $U \in \mathcal U$,
\item $x_{m+1} \in \U_\star(f(\U_\star(x)))$, and
\item for all $i$ with $m < i < n$, we have $x_{i+1} \in \mathcal U_\star(f(\mathcal U_\star(x_i) \cap X))$.
\end{enumerate}
A sequence $\seq{x_n}{n < \w}$ is \emph{eventually decomposable} into $\mathcal U$-compliant $x$-loops if there is some increasing sequence $\seq{n_k}{k < \w}$ of natural numbers such that $\seq{x_i}{n_k < i \leq n_{k+1}}$ is a $\mathcal U$-compliant $x$-loop for every $k$ (the ``eventually'' in the name refers to the fact that we do not require $n_0 = 0$).

The following three ``lemmas'' have trivial proofs that amount simply to checking the definitions involved, but we record them here as small steps toward the proof in the next section. For each lemma, suppose $X$ is a closed subset of $[0,1]^\dlt$, $f: X \to X$ is continuous, and $x \in X$.

\begin{lemma}\label{lem:looprefinement}
If $\mathcal U$ and $\mathcal V$ are nice open covers of $X$ and $\mathcal U$ refines $\mathcal V$, then every $\mathcal U$-compliant $x$-loop is also a $\mathcal V$-compliant $x$-loop.
\end{lemma}

\begin{lemma}\label{lem:loopconcatonation}
Let $\mathcal U$ be a nice open cover of $X$ and let $\seq{x_n}{n < \w}$ be a sequence of points that is eventually decomposable into $\mathcal U$-compliant $x$-loops. Then $\seq{x_n}{n < \w}$ is eventually compliant with $\mathcal U$.
\end{lemma}

\begin{lemma}\label{lem:coordinates}
Let $\mathcal U$ be a nice open cover of $X$, and let $F$ denote the finite set of ordinals used in the definition of $\U$. Suppose $\seq{x_i}{m < i \leq n}$ and $\seq{y_i}{m < i \leq n}$ are two sequences of points in $[0,1]^\dlt$, and that $\pi_\a(x_i) = \pi_\a(y_i)$ for all $i \leq n$ and $\a \in F$. Then $\seq{x_i}{m < i \leq n}$ is a $\U$-compliant $x$-loop if and only if $\seq{y_i}{m < i \leq n}$ is.
\end{lemma}

The observant reader will notice that these sorts of loops appeared already in our proof of the Bowen-Sharkovsky theorem in Section~\ref{sec:background}. We now state the (already proved!) stronger version of Theorem~\ref{thm:BS} that will be used in the proof of the main theorem:

\begin{corollary}\label{cor:BS}
Let $(X,f)$ be a weakly incompressible metrizable dynamical system, and fix $x \in X$. There is a sequence $\seq{x_n}{n < \w}$ of points in $X$ such that, for any nice open cover $\U$ of $X$, $\seq{x_n}{n < \w}$ is eventually decomposable into $\U$-compliant $x$-loops.
\end{corollary}

\section{The main theorem}\label{sec:main}

Before beginning the proof of our main theorem, we briefly review the basic theory of elementary submodels, as these will be the main tool for guiding our construction. We recommend \cite{Hdg} for a more thorough treatment of the topic.

Given a large, uncountable set $H$, we will consider the structure $(H,\in)$. $M \sub H$ is an \emph{elementary submodel} of $H$ if, given any formula $\varphi$ of first-order logic and any $a_1,a_2,\dots,a_n \in M$,
$$(H,\in) \models \varphi(a_1,a_2,\dots,a_n) \ \ \ \Leftrightarrow \ \ \ (M,\in) \models \varphi(a_1,a_2,\dots,a_n)$$
In other words, $M$ and $H$ agree with each other on every first-order statement that can be formulated within $M$.

For our proof, $H$ will be taken to be the set of all sets hereditarily smaller than $\k$ for some sufficiently large regular cardinal $\k$. The structure $(H,\in)$ satisfies all the axioms of $\zfc$ except for the power set axiom, and even this fails only for sets $X$ with $|2^X| \geq \k$. This makes $H$ a good substitute for the universe of all sets. Indeed, if $\k$ is larger than any set mentioned in our proof, then $H$ satisfies $\zfc$ for all practical purposes.

Suppose $M$ is an elementary submodel of $H$. Since $H$ satisfies (most of) $\zfc$, so must $M$. Thus objects definable without parameters, like rational numbers, the ordinals $\w$ and $\w_1$, and topological spaces like $[0,1]$ or $[0,1]^{\w_1}$, are all in $M$. A bit more generally, things definable by formulas with parameters in $M$ are in $M$. For example, if $U$ is a basic open subset of $[0,1]^{\w_1}$ and the ordinals used in the definition of $U$ are all in $M$, then $U \in M$; if $\mathcal U$ is a nice open cover of some $X \in M$, and each $U \in \mathcal U$ is defined using ordinals in $M$, then $\mathcal U \in M$. 

The existence of elementary submodels of $H$ is guaranteed by the downward L\"{o}wenheim-Skolem theorem (see chapter 3 of \cite{Hdg}). We will use the following version of this theorem to facilitate our proof:

\begin{lemma}[L\"{o}wenheim-Skolem]\label{lem:LS}
Let $H$ be an uncountable set, and let $A \sub H$ be countable. There exists a sequence $\seq{M_\a}{\a < \w_1}$ of elementary submodels of $H$ such that
\begin{enumerate}
\item $A \sub M_0$, and $M_\b \sub M_\a$ whenever $\b < \a$.
\item each $M_\a$ is countable.
\item for limit $\a$, $M_\a = \bigcup_{\b < \a}M_\b$.
\item for each $\a$, $\seq{M_\b}{\b < \a} \in M_{\a+1}$.
\end{enumerate}
\end{lemma}

We are now in a position to prove the main theorem. As mentioned in the introduction, our application of elementarity parallels that in Section 3 of \cite{D&H}. In order to make things easier for the reader (especially the reader already familiar with \cite{D&H}), we have tried to match our notation to theirs wherever possible.

\begin{theorem}[Main Theorem]\label{thm:main}
Suppose $(X,f)$ is a dynamical system with weight $\aleph_1$. Then $(X,f)$ is a quotient of $(\w^*,\s)$ if and only if $f$ is weakly incompressible.
\end{theorem}

\begin{proof}
Every quotient of $(\w^*,\s)$ is weakly incompressible by Lemma~\ref{lem:ct}. We must prove that a weakly incompressible dynamical system with weight $\aleph_1$ is a quotient of $(\w^*,\s)$.

Let $(X,f)$ be a weakly incompressible dynamical system with weight $\aleph_1$. Without loss of generality, suppose $X \sub [0,1]^{\w_1}$ and $\vec{0} \in X$. Recall that $[0,1]^{\w_1}$ is a homogeneous topological space, so $\vec{0} \in X$ really can be assumed without any loss of generality.

Using transfinite recursion, we will construct maps $q_\b: \w \to [0,1]$. In the end, the diagonal mapping $Q = \Delta_{\b < \w_1}q_\b$ will define a sequence $\seq{Q(n)}{n < \w}$ in $[0,1]^{\w_1}$ that is eventually compliant with every nice open cover of $X$. By Lemma~\ref{lem:mainlemma}, this will be enough to prove the theorem.

The recursion will be guided by a sequence of elementary submodels as described in Lemma~\ref{lem:LS}. Fix $\k$ sufficiently large, let $H$ denote the set of all sets hereditarily smaller than $\k$, and fix a sequence $\seq{M_\a}{\a < \w_1}$ of countable elementary submodels of $H$ such that
\begin{enumerate}
\item $X,f \in M_0$.
\item $M_\b \sub M_\a$ whenever $\b < \a$.
\item for limit $\a$, $M_\a = \bigcup_{\b < \a}M_\b$.
\item for each $\a$, $\seq{M_\b}{\b < \a} \in M_{\a+1}$.
\end{enumerate}
For each $\a < \w_1$, define $\dlt_\a = \w_1 \cap M_\a$. It can be shown that if $\b$ is a countable ordinal in $M_\a$, then every ordinal less than $\b$ is also in $M_\a$. It follows that $\dlt_\a$ is a countable ordinal, namely the supremum of all countable ordinals in $M_\a$.

For each $\a < \w_1$, let $X_\a = \Pi_{\dlt_\a}(X)$. For every $\a$, if $\Pi_{\dlt_\a} (x) = \Pi_{\dlt_\a} (y)$, then $\Pi_{\dlt_\a} \circ f(x) = \Pi_{\dlt_\a} \circ f(y)$. This follows from the proof of Lemma~\ref{lem:projections}: the function $\phi$ defined there can be defined inside $M_\a$, so that $\dlt_\a$ must be closed under $\phi$, which means $\dlt_\a \in C$.

Thus we may define $f_\a: X_\a \to X_\a$ to be the unique self-map of $X_\a$ satisfying $\Pi_{\dlt_\a} \circ f = f_\a \circ \Pi_{\dlt_\a}$, namely $f_\a(x) = \Pi_{\dlt_\a}(f(\Pi_{\dlt_\a}^{-1}(x)))$. This function is continuous by Lemma~\ref{lem:continuity}.

$(X_\a,f_\a)$ is a dynamical system, and $\Pi_{\dlt_\a}$ provides a natural quotient mapping from $(X,f)$ to $(X_\a,f_\a)$. $X_\a$ is metrizable because it is a subset of $[0,1]^{\dlt_\a}$, and $f_\a$ is weakly incompressible by Lemma~\ref{lem:ct} (alternatively, weak incompressibility can be proved directly by an elementarity argument). We may think of the $(X_\a,f_\a)$ as metrizable ``reflections'' of $(X,f)$.

If $\seq{x_n}{n < \w}$ is a sequence of points in $[0,1]^{\dlt_\a}$ for some $\a$, let us say that a sequence $\seq{y_n}{n < \w}$ of points in $[0,1]^{\w_1}$ is a \emph{lifting} of $\seq{x_n}{n < \w}$ if $\Pi_{\dlt_\a}(y_n) = x_n$ for all $n$ (and similarly for finite sequences of points).

We are now in a position to begin our recursive construction of the maps $q_\a$. Let $\U$ be a nice open cover of $X$ with $\U \in M_0$. Only ordinals less than $\dlt_0$ can be used in the definition of $\U$, so $\U$ naturally projects to a nice open cover of $X_0$ in $[0,1]^{\dlt_0}$, namely
$$\Pi_{\dlt_0}(\U) = \set{\Pi_{\dlt_0}(U)}{U \in \U}.$$
Applying Corollary~\ref{cor:BS} to $(X_0,f_0)$, we obtain a sequence $\seq{x_n}{n < \w}$ of points in $X_0$ such that, for any nice open cover $\U$ of $X$ with $\U \in M_0$, $\seq{x_n}{n < \w}$ eventually decomposes into $\Pi_{\dlt_0}(\U)$-compliant $\vec{0}$-loops.

By Lemma~\ref{lem:coordinates}, any lifting of $\seq{x_n}{n < \w}$ to $[0,1]^{\w_1}$ eventually decomposes into $\U$-compliant $\vec{0}$-loops for any $\U \in M_0$.

For $\b < \dlt_0$, define $q_\b(n) = \pi_\b(x_n)$ (in other words, we define the $q_\b$ so that $\Delta_{\b < \dlt_0}q_\b$ maps $\w$ to the sequence just constructed).

Let $\seq{n_k}{k < \w}$ be an increasing sequence of natural numbers such that for any $\U \in M_0$, $\U$ a nice open cover of $X$, and for all but finitely many $k$, any lifting of $\seq{x_i}{n_k < i \leq n_{k+1}}$ is a $\U$-compliant $\vec{0}$-loop.

For the successor stage of the recursion, let $\a < \w_1$ and suppose the functions $q_\b$ have already been constructed for every $\b < \dlt_\a$. Let $Q_\a = \Delta_{\b < \dlt_\a}q_\b$, and suppose the following three inductive hypotheses hold:
\begin{itemize}
\item [(H1)] $Q_\a \in M_{\a+1}$.
\item [(H2)] $Q_\a(n_k) = \vec{0}$ for all $k < \w$.
\item [(H3)] For any nice open cover $\mathcal U$ of $X$ with $\U \in M_\a$ and for all but finitely many $k$, any lifting of $\seq{Q_\a(n)}{n_k < i \leq n_{k+1}}$ is a $\mathcal U$-compliant $\vec{0}$-loop.
\end{itemize}
We will show how to obtain $q_\b$ for $\dlt_\a \leq \b < \dlt_{\a+1}$.

Because $M_{\a+1}$ is countable, there are only countable many nice open covers of $X$ in $M_{\a+1}$, namely those that are definable from ordinals less than $\dlt_{\a+1}$. Also, any two nice open covers of $X$ in $M_{\a+1}$, say $\mathcal V$ and $\mathcal W$, have a common refinement that is also a nice open cover of $X$ in $M_{\a+1}$; e.g., one such common refinement is
$$\set{V \cap W}{V \in \mathcal V, W \in \mathcal W, \text{ and }V \cap W \cap X \neq \0}.$$
Thus we may find a countable sequence $\seq{\mathcal U_m}{m < \w}$ of nice open covers of $X$ such that
\begin{enumerate}
\item $\U_m \in M_{\a+1}$ for every $m$,
\item $\mathcal U_n$ refines $\mathcal U_m$ whenever $m \leq n$, and
\item if $\mathcal U$ is any nice open cover of $X$ in $M_{\a+1}$, then $\mathcal U_m$ refines $\mathcal U$ for some $m$.
\end{enumerate}

Fix $m \in \w$ and consider $\mathcal U_m$. The set of ordinals used in the definition of $\mathcal U_m$ is finite and may be split into two parts: those ordinals that are below $\dlt_\a$, which we call $F_m^0$, and those that are in the interval $[\dlt_\a,\dlt_{\a+1})$, which we call  $F_m^1$. The ordinals $F_m^1$ are not in $M_\a$, but we may use elementarity to find a finite set of ordinals $G_m$ in $M_\a$ that reflects the set $F_m^1$.

More formally, suppose that we write down in the language of first-order logic a (very long) formula $\varphi^m$ that does all of the following:
\begin{enumerate}
\item $\varphi^m$ defines $\U_m$ in terms of $F_m^0 \cup F_m^1$,
\item $\varphi^m$ asserts that $\U_m$ is a nice open cover of $X$,
\item $\varphi^m$ records information about how $\U_m$ interacts with $X$ and $f$:
\begin{enumerate}
\item for all $\mathcal J \sub \U_m$, $\varphi^m$ asserts either that $\bigcap \mathcal J \cap X = \0$ or that $\bigcap \mathcal J \cap X \neq \0$,
\item if $\mathcal J \sub \U_m$, $\bigcap \mathcal J \cap X \neq \0$, and $U \in \U_m$, then $\varphi^m$ asserts either that $f(\bigcup \mathcal J \cap X) \cap U = \0$ or that $f(\bigcup \mathcal J \cap X) \cap U \neq \0$.
\end{enumerate}
\end{enumerate}
Given a finite sequence of points, the information contained in $(1)$ is enough to determine precisely which elements of $\U_m$ contain each member of the sequence. Once that is known, the information in $(3)$ is enough to determine whether or not that sequence is a $\U_m$-compliant $\vec{0}$-loop. 

By elementarity, there is a finite set $G_m$ of ordinals in $M_\a$ such that $\varphi^m$ remains true when the members of $F_m^1$ are replaced with the members of $G_m$. For each $\b \in F_m^1$, let $\b_m$ denote the corresponding member of $G_m$.

Let $\mathcal V_m$ be the nice open cover of $X$ that is defined via $\varphi^m$, but substituting the members of $G_m$ in place of the corresponding members of $F_m^1$. We think of $\V_m$ as the reflection of $\U_m$ in $M_\a$. Let $k(m) \in \w$ be the least natural number with the property that for all $k \geq k(m)$, any lifting of $\seq{Q_\a(i)}{n_{k} < i \leq n_{k+1}}$ is a $\V_m$-compliant $\vec{0}$-loop. This $k(m)$ exists by our third inductive hypothesis. If $m < n$, then $\V_n$ refines $\V_m$, so $k(m) \leq k(n)$ by Lemma~\ref{lem:looprefinement}.

We are now in a position to define the maps $q_\b$ for $\dlt_\a \leq \b < \dlt_{\a+1}$:
$$
q_\b(i) = \begin{cases} 
0 & \textrm{ if $n_{k(m)} < i \leq n_{k(m+1)}$ and $\b \notin F_m^1$,} \\
q_{\b_m}(i) & \textrm{ if $n_{k(m)} < i \leq n_{k(m+1)}$ and $\b \in F_m^1$.}
\end{cases}
$$
Roughly, this says that $q_\b$ assumes the behavior of its mirror image $q_{\b_m}$ on the interval between $n_{k(m)}$ and $n_{k(m+1)}$, provided some suitable mirror image has already been found. As $m$ increases, the $\b_m$ become better and better reflections of $\b$, because the formulas $\varphi^m$ include more and more information about $X$ and $f$.

With the $q_\b$ thus defined, we need to check that our three inductive hypotheses remain true at the next stage of the recursion. For the first hypothesis, note that, because we have $\seq{M_\b}{\b < \a+1} \in M_{\a+2}$, the construction of the $q_\b$, $\dlt_\a \leq \b < \dlt_{\a+1}$, can be carried out in $M_{\a+2}$. Thus the result of this construction, namely $Q_{\a+1} = \Delta_{\b < \dlt_{\a+1}}q_\b$, is a member of $M_{\a+2}$, as desired. The second inductive hypothesis, that $Q_{\a+1}(n_k) = \vec{0}$ for all $k$, is clear from the definition of the $q_\b$.

For the third inductive hypothesis, let $\mathcal U$ be a nice open cover of $X$ with $\U \in M_{\a+1}$ and fix $m$ large enough so that $\mathcal U_m$ refines $\mathcal U$. By Lemma~\ref{lem:looprefinement}, it is enough to check that for all but finitely many $k$, any lifting of $\seq{Q_{\a+1}(i)}{n_k \leq i \leq n_{k+1}}$ is a $\mathcal U_m$-compliant $\vec{0}$-loop.

By the definition of $k(m)$, if $k(m) \leq k < k(m+1)$ then any lifting of $\seq{Q_{\a}(i)}{n_k < i \leq n_{k+1}}$ is a $\mathcal V_m$-compliant $\vec{0}$-loop. Of course, by Lemma~\ref{lem:coordinates} only the coordinates in $F_m^0 \cup G_m$ are relevant to determining this fact. More specifically, in $M_\a$ it may be proved that any sequence $\seq{x(i)}{n_k < i \leq n_{k+1}}$ agreeing with $\seq{Q_{\a}(i)}{n_k < i \leq n_{k+1}}$ on the members of $F_m^0 \cup G_m$ is a $\vec{0}$-loop compliant with the nice open cover of $X$ defined by $\varphi^m$ using $F_m^0 \cup G_m$, namely $\V_m$. By our definition of the $q_\b$, $\seq{Q_{\a+1}(i)}{n_k < i \leq n_{k+1}}$ is such a sequence, except that we have replaced the members of $G_m$ with the members of $F_m^1$. By elementarity and our choice of the $G_m$, if $k(m) \leq k < k(m+1)$ then any lifting of $\seq{Q_{\a}(i)}{n_k < i \leq n_{k+1}}$ is a $\vec{0}$-loop compliant with the nice open cover of $X$ defined by $\varphi^m$ using $F_m^0 \cup F_m^1$, namely $\U_m$.

Given any $k \geq k(m)$, the same argument shows that if $\ell$ is the natural number with $k(\ell) \leq k < k(\ell+1)$, then $\seq{Q_{\a}(i)}{n_k < i \leq n_{k+1}}$ is a $\U_\ell$-compliant $\vec{0}$-loop. By Lemma~\ref{lem:looprefinement} and the fact that $\U_\ell$ refines $\mathcal U_m$, $\seq{Q_{\a+1}(i)}{n_k \leq i \leq n_{k+1}}$ is a $\mathcal U_m$-compliant $\vec{0}$-loop for all $k \geq k(m)$. This proves the third inductive hypothesis and completes the successor step of our recursion.

At limit stages there is nothing to construct: due to our choice of the $M_\a$, we have $\dlt_\a = \bigcup_{\b < \a}\dlt_\b$ for limit $\a$, so that all the $q_\b$, $\b < \dlt_\a$, have already been defined by stage $\a$. We only need to check for limit $\a$ that our inductive hypotheses remain true. The first hypothesis is true because $\seq{M_\b}{\b < \a} \in M_{\a+1}$. The second hypothesis is true at $\a$ if it is true at every $\b < \a$.

To check the third hypothesis, suppose $\mathcal U$ is a nice open cover of $X$ with $\U \in M_\a$. $\mathcal U$ is defined using only finitely many ordinals less than $\dlt_\a$, so $\mathcal U \in M_\b$ already for some $\b < \a$. At stage $\b$, we ensured that any lifting of $\seq{Q_\b(i)}{n_k \leq i \leq n_{k+1}}$ is a $\mathcal U$-compliant $\vec{0}$-loop for all but finitely many $k$. But $Q_\a$ agrees with $Q_\b$ on all coordinates below $\dlt_\b$, so any lifting of $\seq{Q_\a(i)}{n_k \leq i \leq n_{k+1}}$ is also a lifting of $\seq{Q_\b(i)}{n_k \leq i \leq n_{k+1}}$, and is therefore a $\mathcal U$-compliant $\vec{0}$-loop. This completes our recursion.

We claim that the map $Q = \Delta_{\a < \w_1}q_\a$ is as required; i.e., the sequence $\seq{Q(n)}{n < \w}$ is eventually compliant with every nice open cover of $X$. Indeed, if $\mathcal U$ is a nice open cover of $X$, then $\mathcal U$ is defined by finitely many ordinals, so it was considered at some stage $\a$ of our recursion. At that stage we guaranteed that any lifting of $\seq{Q_\a(n)}{n < \w}$ is eventually decomposable into $\mathcal U$-compliant $\vec{0}$-loops. $\seq{Q(n)}{n < \w}$ is such a lifting, so it is eventually compliant with $\U$.
\end{proof}

\section{Related results}\label{sec:MA}

\subsection*{Two corollaries}

Consider the following two theorems, both discussed in the introduction:
\begin{itemize}
%\item (Bowen-Sharkovsky, \cite{Bwn} \cite{Srk}) \emph{A metrizable dynamical system is an abstract $\w$-limit set if and only if it is weakly incompressible.}
\item (Parovi\v{c}enko, \cite{Par}) \emph{Every compact Hausdorff space of weight $\aleph_1$ is a continuous image of $\w^*$.}
\item (Dow-Hart, \cite{D&H}) \emph{Every connected compact Hausdorff space of weight $\aleph_1$ is a continuous image of $\mathbb H^*$, where $\mathbb H = [0,\infty)$.}
\end{itemize}
We begin this section by showing that both of these theorems can be derived as fairly straightforward consequences of Theorem~\ref{thm:main}.

\begin{lemma}\label{lem:counterexample}
Let $Y$ be a compact Hausdorff space of weight $\k$. There is a weakly incompressible dynamical system $(X,f)$ such that $X$ also has weight $\k$ and $Y$ is clopen in $X$.
\end{lemma}
\begin{proof}
Let $Y$ be a compact Hausdorff space of weight $\k$. Let $X$ be the one-point compactification of $\Z \times Y$, where $\Z$ is given the discrete topology. Let $*$ denote the unique point of $X - \Z \times Y$, and define $f: X \to X$ so that $f(*) = *$, and $f(n,y) = (n+1,y)$. Clearly, $f$ is continuous, $X$ has weight $\k$, and $Y$ is (homeomorphic to) a clopen subset of $X$.

It remains to show that $(X,f)$ is chain transitive. Let $\mathcal U$ be any open cover of $X$ and $a,b \in X$. To find a $\mathcal U$-chain from $a$ to $b$, fix $U \in \mathcal U$ with $* \in U$. If $a = *$ and $b = (n,y)$, we may choose $m$ small enough that $m < n$ and $(m,y) \in U$. Then
$$\<*,(m,y),(m+1,y),\dots,(n,y)\>$$
is a $\mathcal U$-chain from $a$ to $b$. Similarly if $a = (m,y)$ and $b = *$, choose $n$ large enough that $n > m$ and $(n,y) \in U$. Then
$$\<(m,y),(m+1,y),\dots,(n,y),*\>$$
is a $\mathcal U$-chain from $a$ to $b$. If $a \neq * \neq b$, then we may get a $\mathcal U$-chain from $a$ to $b$ by concatonating a $\U$-chain from $a$ to $*$ with a $\U$-chain from $*$ to $b$. Thus $(X,f)$ is chain transitive.
\end{proof}

Parovi\v{c}enko's theorem follows immediately from Theorem~\ref{thm:main} and the next result:

\begin{proposition}\label{prop:counterexample}
Suppose every weakly incompressible dynamical system of weight $\k$ is a quotient of $(\w^*,\s)$. Then every compact Hausdorff space of weight $\k$ is a continuous image of $\w^*$.
\end{proposition}
\begin{proof}
Suppose every weakly incompressible dynamical system of weight $\k$ is a quotient of $(\w^*,\s)$, and let $Y$ be a compact Hausdorff space of weight $\k$. Let $(X,f)$ be the dynamical system guaranteed by Lemma~\ref{lem:counterexample}. $(X,f)$ is a quotient of $(\w^*,\s)$, so in particular there is a continuous surjection $Q: \w^* \to X$. The pre-image of $Y$ is clopen in $\w^*$, and therefore homeomorphic to $\w^*$. The restriction of $Q$ to $Q^{-1}(Y)$ provides a continuous surjection from (a copy of) $\w^*$ to $Y$.
\end{proof}

Observe that a compact Hausdorff space $X$ is connected if and only if $(X,\mathrm{id})$ is a weakly incompressible dynamical system. With this in mind, Theorem~\ref{thm:main} and the following proposition immediately imply the theorem of Dow and Hart:

\begin{proposition}
If $(X,\mathrm{id})$ is a quotient of $(\w^*,\s)$ then $X$ is a continuous image of $\mathbb H^*$.
\end{proposition}
\begin{proof}
Suppose $(X,\mathrm{id})$ is a quotient of $(\w^*,\s)$, and assume that $X \sub [0,1]^\dlt$ for some $\dlt$. By Theorem~\ref{thm:main} and the second part of Lemma~\ref{lem:mainlemma}, there is a sequence $\seq{x_n}{n < \w}$ of points in $[0,1]^\dlt$ that is eventually compliant with every nice open cover of $X$.

Define a map $q: \mathbb H \to [0,1]^\dlt$ by sending $n$ to $x_n$ for each integer $n$, and then extending $q$ linearly to the rest of $\mathbb H$. This function induces a map $Q: \mathbb H^* \to [0,1]^\dlt$, and we claim that $Q$ is a continuous surjection from $\mathbb H^*$ to $X$.

$Q$ is continuous by definition. We see that $Q(\mathbb H^*) \supseteq X$ by considering those elements of $\mathbb H^*$ that are supported on the integers. It remains to show $Q(\mathbb H^*) \sub X$. Let $W$ be an open set containing $X$ and let $\U$ be a nice open cover with $\bigcup \U \sub W$. Let $\V$ be a star refinement of a star refinement of $\U$. Because $\seq{x_n}{n < \w}$ is eventually compliant with $\V$, there is some $m$ such that for all $n \geq m$, $x_{n+1} \in \V_\star(\V_\star(x_n))$. By our choice of $\V$, there is some $U \in \U$ with $x_n,x_{n+1} \in U$. As every basic open subset of $[0,1]^\dlt$ is convex, $q(r) \in U$ for all $r \in [x_n,x_{n+1}]$. Thus $q(r) \in W$ for every $r \in [m,\infty)$, which implies $Q(\mathbb H^*) \sub W$. Since $W$ was arbitrary, $Q(\mathbb H^*) \sub X$.
\end{proof}

%We list one more corollary of Theorem~\ref{thm:main}, obtained by Stone duality from the special case when $X$ is zero-dimensional and $f$ is an autohomeomorphism of $X$. Let $s$ denote the automorphism of $\pwmf$ induced by the successor map on $\w$.

%\begin{proposition}
%Let $B$ be a Boolean algebra of cardinality $\aleph_1$, and let $\varphi$ be an automorphism of $B$. The following are equivalent:
%\begin{enumerate}
%\item there is an embedding $e: B \to \pwmf$ with $e \circ \varphi = s \circ e$.
%\item $\varphi(b) \not\leq b$ for all $b \in B - \{0,1\}$.
%\end{enumerate}
%\end{proposition}

\subsection*{The first and fourth heads of $\b\w$}

If we assume the Continuum Hypothesis, then Theorem~\ref{thm:main} gives a complete internal characterization of the quotients of $(\w^*,\s)$:

\begin{theorem}\label{thm:ch}
Assuming $\ch$, the following are equivalent:
\begin{enumerate}
\item $(X,f)$ is a quotient of $(\w^*,\s)$.
\item $X$ has weight at most $\continuum$ and $f$ is weakly incompressible.
\item $X$ is a continuous image of $\w^*$ and $f$ is weakly incompressible.
\end{enumerate}
\end{theorem}
\begin{proof}
The equivalence of $(1)$ and $(2)$ is a straightforward consequence of Theorem~\ref{thm:main} and $\ch$. The equivalence of $(2)$ and $(3)$ is a straightforward consequence of Parovi\v{c}enko's characterization of the continuous images of $\w^*$ under $\ch$.
\end{proof}

Of the six implications this theorem entails, three are provable from $\zfc$: $(1) \Rightarrow (2)$, $(1) \Rightarrow (3)$, and $(3) \Rightarrow (2)$. We will now consider the other three, and show that each of them is independent of $\zfc$.

Lemma~\ref{lem:counterexample} shows that $(2) \Rightarrow (3)$ if and only if every compact Hausdorff space of weight $\leq \continuum$ is a continuous image of $\w^*$. This is a purely topological question about $\w^*$ that is considered elsewhere, e.g. in \cite{JvM}. It is known to be independent: for example, a result of Kunen states that $\w_2+1$ is not a continuous image of $\w^*$ in the Cohen model.

Because $(1) \Rightarrow (3)$ is a theorem of $\zfc$, the previous paragraph also shows that $(2) \Rightarrow (1)$ is independent.

The independence of $(3) \Rightarrow (1)$ requires a different argument. Consider the following corollary to Theorem~\ref{thm:ch}:

\begin{corollary}\label{cor:quotient}
Assuming $\ch$, $(\w^*,\s^{-1})$ is a quotient of $(\w^*,\s)$.
\end{corollary}
\begin{proof}
The proof is immediate from Theorem~\ref{thm:ch} and the following observation: \emph{If $X$ is a compact Hausdorff space and $f: X \to X$ is a homeomorphism, then $f$ is weakly incompressible if and only if $f^{-1}$ is.}

This is easy to see using chain transitivity: given an open cover $\U$ of $X$ and any $a,b \in X$, $(X,f)$ has a $\U$-chain from $a$ to $b$ if and only if $(X,f^{-1})$ has a $\U$-chain from $b$ to $a$.
\end{proof}

To show that $(3) \Rightarrow (1)$ is independent, it is enough to prove that the conclusion of Corollary~\ref{cor:quotient} is independent.

\begin{theorem}\label{thm:farah}
Assuming $\ocama$, $(\w^*,\s^{-1})$ is not a quotient of $(\w^*,\s)$.
\end{theorem}

Recall that a continuous function $F: \w^* \to \w^*$ is \emph{trivial} if there is a function $f: \w \to \b\w$ such that $F = \b f \rest \w^*$. Similarly, $F: A^* \to \w^*$ is trivial if it is induced by a function $A \to \b\w$. To prove Theorem~\ref{thm:farah}, we will use a deep theorem greatly restricting the kinds of self-maps of $\w^*$ we find under $\ocama$. A very general version of the result is proved by Farah in \cite{IF1}, but we need only a special case, which is already implicit in the work of Velickovic \cite{Vel}, and has precursors in the work of Shelah-Stepr\={a}ns \cite{S&S} and Shelah \cite{SSh}.

\begin{theorem}[Farah, et al.]\label{thm:ocama}
Assuming $\ocama$, for any continuous $F: \w^* \to \w^*$ there is some $A \sub \w$ such that $F \rest A^*$ is trivial and $F(\w^* - A^*)$ is nowhere dense.
\end{theorem}

\begin{proof}[Proof of Theorem~\ref{thm:farah}]
Suppose $Q$ is a quotient mapping from $(\w^*,\s)$ to $(\w^*,\s^{-1})$. Using Theorem~\ref{thm:ocama}, fix $A \sub \w$ such that $Q \rest A^*$ is trivial and $Q(\w^* - A^*)$ is nowhere dense. Also, fix $q: A \to \b\w$ such that $Q \rest A^* = \b q \rest A^*$.

Because $Q$ is surjective, $A$ must be infinite.

Let $X = \set{a \in A}{q(a) \in \w}$. Observe that $Q \rest X$ remains trivial and that $Q(\w^* - X^*)$ remains nowhere dense. Thus, replacing $A$ with $X$ if necessary, we may (and do) assume that $q(a) \in \w$ for all $a \in A$.

If $q$ is not finite-to-one on $A$, there is an infinite set $X \sub A$ and some $n \in \w$ with $q(X) = n$, but then $Q(p) = n$ for any $p \in X^*$, a contradiction. Thus $q$ is finite-to-one on $A$.

Suppose $A$ is not co-finite. Then
$$B = \set{a \in A}{a+1 \notin A}$$
is infinite. Using basic facts about Stone extensions, $\s^{-1} \circ Q(B^*) = (q(B)-1)^*$. This set is clopen (in particular, it has nonempty interior), so we may find some $p \in B^*$ such that $\s^{-1} \circ Q (p) \notin Q(\w^* - A^*)$. However, $Q \circ \s(p) \in Q \circ \s(B^*) = Q((B+1)^*) \sub Q(\w^* - A^*)$, so that $\s^{-1} \circ Q(p) \neq Q \circ \s(p)$, a contradiction. Thus $A$ is co-finite, and $Q = \b q \rest \w^*$ for some finite-to-one function $q: A \to \w$. Since changing $q$ on a finite set does not change $Q = \b q \rest \w^*$, we may assume $A = \w$, and $Q$ is induced by a finite-to-one function $q: \w \to \w$.

We now construct an infinite sequence of natural numbers as follows. Pick $b_0 \in \w$ arbitrarily. Assuming $b_0,b_1,\dots,b_n$ are given, there are co-finitely many $b \in \w$ satisfying
\begin{enumerate}
\item $b \neq b_0,b_1,\dots,b_n$,
\item $q(b)-1 \neq q(b_0+1),q(b_1+1),\dots,q(b_n+1)$, and
\item $q(b+1) \neq q(b_0)-1,q(b_1)-1,\dots,q(b_n)-1$.
\end{enumerate}
Also, a straightforward argument by contradiction shows that there are infinitely many $b \in \w$ satisfying
\begin{enumerate}
\setcounter{enumi}{3}
\item $q(b_{n+1})-1 \neq q(b_{n+1}+1)$.
\end{enumerate}
Thus we may choose some $b_{n+1} \in \w$ satisfying $(1) - (4)$.

Let $B = \set{b_n}{n < \w}$ and let $p \in B^*$. Then $Q \circ \s(p) \in q(B+1)^*$ and $\s^{-1} \circ Q(p) \in (q(B) - 1)^*$. By construction, $q(B+1) \cap (q(B) - 1) = \0$, which shows $Q \circ \s(p) \neq \s^{-1} \circ Q(p)$.
\end{proof}

We do not know whether Corollary~\ref{cor:quotient} can be improved from a quotient mapping to an isomorphism:

\begin{question}
Is it consistent that there is a homeomorphism $H: \w^* \to \w^*$ with $H \circ \s = \s^{-1} \circ H$?
\end{question}

Observe that our proof of Theorem~\ref{thm:main} cannot produce a homeomorphism: in the quotient mapping constructed there, the inverse image of $\vec{0}$ has nonempty interior. Therefore some new idea would be needed to answer this question in the affirmative. We point out that if the answer to this question is yes, then it seems likely that $\ch$ will imply the existence of such an isomorphism already (see Section 5.1 of \cite{IF2}). See \cite{SG2} for some partial results. 

\subsection*{An extension using Martin's Axiom}

We end with an extension of Theorem~\ref{thm:main} to cardinals $\k < \pseudo$.

\begin{theorem}\label{thm:MA}
Let $(X,f)$ be a dynamical system with the weight of $X$ less than $\pseudo$. Then $(X,f)$ is a quotient of $(\w^*,\s)$ if and only if $f$ is weakly incompressible.
\end{theorem}
\begin{proof}
Let $(X,f)$ be a weakly incompressible dynamical system, and let $\k$ be the weight of $X$. Suppose $\k < \pseudo$. By a theorem of M. Bell, this is equivalent to assuming $\mask$, Martin's Axiom at $\k$ for $\s$-centered posets. We may (and do) assume that $X \sub [0,1]^\k$.

We will use $\mask$ to construct a sequence of points in $[0,1]^\k$ that is eventually compliant with every nice open cover of $X$.

Recall that $[0,1]^\k$ is separable, and fix a countable dense $D \sub [0,1]^\k$. Let us assume that $X$ is nowhere dense in $[0,1]^\k$ and that $X \cap D = \0$. This assumption does not sacrifice any generality, since we could always just replace $[0,1]^\k$ with $[0,1] \times [0,1]^\k$, identify $X$ with $\{0\} \times X$, and replace $D$ with $(\Q \cap (0,1]) \times D$.

Fix $x \in X$. Let $\p$ be the set of all pairs $\<s,\U\>$, such that $s$ is a sequence of distinct points in $D$ and $\U$ is a nice open cover of $X$. Order $\p$ by defining $\<t,\V\> \leq \<s,\U\>$ if and only if
\begin{itemize}
\item $s$ is an initial segment of $t$.
\item $\V$ refines $\U$.
\item either $t = s$, or $t-s$ is a $\U$-compliant $x$-loop.
\end{itemize}
Ultimately, we will use $\mask$ to obtain a suitably generic $G \sub \p$, and then $\g = \bigcup \set{s}{\<s,\U\> \in G}$ will be the desired sequence of points. Roughly, a condition $\<s,\U\>$ is a promise that $s$ is an initial segment of $\g$, and that the part of $\g$ after $s$ will decompose into $\U$-compliant $x$-loops.

$\p$ is clearly reflexive, and if $\<u,\mathcal W\> \leq \<t,\V\> \leq \<s,\U\>$ then $\<u,\mathcal W\> \leq \<s,\U\>$ by Lemma~\ref{lem:looprefinement}. Thus $\p$ is a pre-order, and it makes sense to talk about forcing with $\p$.

Because $D$ is countable, there are only countably many possibilities for the first coordinate of a condition in $\p$. To show that $\p$ is $\s$-centered, it suffices to show that if two conditions $\<s,\U\>$, $\<s,\V\>$ have the same first coordinate $s$, then they have a common extension. Taking $\mathcal W$ to be any nice open cover of $X$ that refines both $\U$ and $\V$ (for example $\mathcal W = \set{U \cap V}{U \in \U, V \in \V, \text{ and } U \cap V \cap X \neq \0}$), then $\<s,\mathcal W\> \leq \<s,\U\>$ and $\<s,\mathcal W\> \leq \<s,\V\>$. Thus $\p$ is $\s$-centered.

If $\U$ is a nice open cover of $X$, define
$$D_{\U} = \set{\<s,\V\> \in \p}{\V \text{ refines } \U}.$$
We claim that $D_\U$ is dense in $\p$. To see this, fix a nice open cover $\U$ of $X$ and let $\<s,\V\> \in \p$. Clearly $\<s,\U\> \in \p$, and we have already seen (in the previous paragraph) that any two conditions in $\p$ with the same first coordinate have a common extension. This common extension is in $D_{\U}$ and below $\<s,\V\>$, so $D_\U$ is dense in $\p$.

By $\mask$, there is a filter $G$ on $\p$ such that $D_{\U} \cap G \neq \0$ for every nice open cover $\U$ of $X$. Let $\g = \bigcup \set{s}{\<s,\U\> \in G}$. For any nice open cover $\U$ of $X$, $\g$ is eventually compliant with $\U$ precisely because $G \cap D_{\U} \neq \0$. An application of Lemma~\ref{lem:mainlemma} completes the proof.
\end{proof}

A topic left open by Theorems \ref{thm:main} and \ref{thm:MA} is how to construct quotients or isomorphisms from $(\w^*,\s)$ to dynamical systems of weight $\continuum$ when $\ch$ fails. The following question is a particularly interesting possibility related to the Katowice problem:

\begin{question}
Is it consistent to have a weakly incompressible autohomeomorphism of $\w_1^*$?
\end{question}

If $F$ were such a map, then $F$ cannot be trivial on any set of the form $A^*$, with $A$ co-countable. It is consistent that no such map exists, but it is not currently known whether the opposite is also consistent. See \cite{L&M} for some discussion of this problem and related results.

\end{document}